\documentclass[reqno]{amsart}

\usepackage{amsmath, amssymb, amsfonts}
\usepackage{mathrsfs}
\usepackage{hyperref}

\newtheorem{theorem}{Theorem}[section]
\newtheorem{corollary}[theorem]{Corollary}
\newtheorem{lemma}[theorem]{Lemma}
\newtheorem{proposition}[theorem]{Proposition}
\theoremstyle{remark}
\newtheorem{remark}[theorem]{Remark}
\theoremstyle{definition}
\newtheorem{define}[theorem]{Definition}

\newtheorem*{Acknowledgement}{Acknowledgement}

\numberwithin{equation}{section}

\newcommand{\CC}{\mathbb C}

\begin{document}

\title[Boundary behaviour of the squeezing function and Fridman invariant]{A note on the boundary behaviour of the squeezing function and Fridman invariant}
\author{Van Thu Ninh\textit{$^{1,2}$}, Anh Duc Mai\textit{$^{3}$}, Thi Lan Huong Nguyen\textit{$^{4}$} and Hyeseon Kim\textit{$^{5}$}} 
\address{Ninh Van Thu}
\address{\textit{$^{1}$}~Department of Mathematics, Vietnam National University at Hanoi, 334 Nguyen Trai, Thanh Xuan, Hanoi, Vietnam}
\address{\textit{$^{2}$}~Thang Long Institute of Mathematics and Applied Sciences,
Nghiem Xuan Yem, Hoang Mai, HaNoi, Vietnam}
\email{thunv@vnu.edu.vn}

\address{Mai Anh Duc}
\address{\textit{$^{3}$}~Faculty of Mathematics Physics and Informatics, Tay Bac University, Quyet Tam, Son La City, Son La, Vietnam}
\email{ducphuongma@gmail.com, maianhduc@utb.edu.vn}

\address{Nguyen Thi Lan Huong}
\address{\textit{$^{4}$}~Department of Mathematics, Hanoi University of Mining and Geology, 18 Pho Vien, Bac Tu Liem, Hanoi, Vietnam}
\email{lanhuongmdc@gmail.com}

\address{Hyeseon Kim}
\address{\textit{$^{5}$}~Research Institute of Mathematics, Seoul National University, 1 Gwanak-ro, Gwanak-gu, Seoul 08826, Republic of Korea}
\email{hop222@snu.ac.kr, hop222@gmail.com}

 
\begin{abstract}
Let $\Omega$ be a domain in $\mathbb C^n$. Suppose that $\partial\Omega$ is smooth pseudoconvex of D'Angelo finite type near a boundary point $\xi_0\in \partial\Omega$ and the Levi form has corank at most $1$ at $\xi_0$. Our goal is to show that if the squeezing function $s_\Omega(\eta_j)$ tends to $1$ or the Fridman invariant $h_\Omega(\eta_j)$ tends to $0$ for some sequence $\{\eta_j\}\subset \Omega$ converging to $\xi_0$, then this point must be strongly pseudoconvex.  
\end{abstract}

\maketitle  

\section{Introduction and the main result}  
The study of biholomorphic invariants has been attracted much attention in the complex differential geometry to enhance the comprehension and application of biholomorphic classification of complex domains. The squeezing function, the Fridman invariant, and the quotient invariant by using the Carath\'{e}odory and Kobayashi-Eisenman volume elements, have received increasing interest as biholomorphic invariants in recent years (see \cite{BK19}, \cite{MV19}, \cite{NV18}, \cite{NN19} and the references therein). We particularly consider both the squeezing function and the Fridman invariant associated to a certain class of pseudoconvex domains in $\mathbb C^n$ in this paper.

Let $\Omega$ be a domain in $\mathbb C^n$ and $p\in\Omega$. For a holomorphic embedding $f\colon\Omega\to\mathbb B^n$ with $f(p)=0$, let us define
\[
s_{\Omega, f}(p):=\sup\left\{r>0\colon \mathbb B(0;r)\subset f(\Omega)\right\},
\]
where $\mathbb B(z_{0}; r)\subset\mathbb C^n$ denotes the complex ball of radius $r$ with center at $z_0$ and $\mathbb B^n$ denotes the complex unit ball $\mathbb B(0;1)$. Then the \textit{squeezing function} $s_{\Omega}: \Omega\to\mathbb R$ is defined in \cite{DGZ12} as
\[
s_{\Omega}(p):=\sup_{f} \left\{s_{\Omega, f}(p)\right\}.
\]
Note that $0 < s_{\Omega}(z)\leq 1$ for any point $z \in \Omega$. 

Next, let us recall the Fridman invariant. Let $\Omega$ be a bounded domain in $\mathbb C^n$ and let $B_\Omega(p,r)$ be the Kobayashi ball around $p$ of radius $r>0$. Let $\mathcal{R}$ be the set of all $r>0$ such that there is a holomorphic embedding $f\colon \mathbb B^n\to \Omega$ with $B_\Omega(p,r)\subset f(\mathbb B^n)$. Note that $\mathcal{R}$ is non-empty (cf. \cite{MV19}). Then the \emph{Fridman invariant} is defined by
\[
h_\Omega(p)=\inf_{r\in \mathcal R}\frac{1}{r}.
\]  

Let $\Omega$ be a bounded domain in $\mathbb C^n$ with smooth boundary $\partial\Omega$ and $\xi_0\in \partial \Omega$. Suppose that $\partial \Omega$ is pseudoconvex of D'Angelo finite type near $\xi_0$. Then it is proved in \cite{DGZ16}, \cite{DFW14} and \cite{KZ16} that $\xi_0$ is strongly pseudoconvex if $\lim\limits_{\Omega \ni z\to \xi_0}s_{\Omega}(z)=1$. 

Now we consider a sequence $\{\eta_j\}\subset \Omega$ converging to $\xi_0$. Suppose that $\Omega$ is pseudoconvex of D'Angelo finite type near $\xi_0$ and $\lim\limits_{j\to \infty} s_\Omega(\eta_j)=1$ or $\lim\limits_{j\to \infty} h_\Omega(\eta_j)=0$. In \cite{JK18} and \cite{MV19}, they proved that if the sequence $\{\eta_j\}\subset \Omega$ converges to $\xi_0$ along the inner normal line to $\partial\Omega$ at $\xi_0$, then $\xi_0$ must be strongly pseudoconvex (for details, see \cite{JK18} for $n=2$ and \cite{MV19} for general case). Moreover, this result was obtained in \cite{Ni18} for the case that $\{\eta_j\}\subset \Omega$ converges nontangentially to $\xi_0$ and in \cite{NN19} for the case that  $\{\eta_j\}\subset \Omega$ converges $\left(\frac{1}{m_1},\ldots, \frac{1}{m_{n-1}}\right)$-nontangentially to an $h$-extendible boundary point $\xi_0$ (for definition, see \cite{NN19}). Here $(1, m_1, \ldots, m_{n-1})$ is the \emph{multitype of $\partial\Omega$ at $\xi_0$} and the \emph{$h$-extendiblility at $\xi_0$} means that the Catlin multitype and D'Angelo multitype of $\partial\Omega$ at $\xi_0$ coincide (see \cite{Yu94}). 

Throughout this paper, we consider a smooth bounded domain $\Omega$ in $\mathbb C^n$ and a point $\xi_0\in \partial\Omega$ such that $\Omega$ is pseudoconvex of D'Angelo finite type near $\xi_0$ and the Levi form has corank at most $1$ at $\xi_{0}$. In this paper, we prove the following theorem.

\begin{theorem}\label{main thm}
Let $\Omega$ be a bounded domain in $\mathbb C^n$ with smooth pseudoconvex boundary. If $\xi_0$ is a boundary point of $\Omega$ of D'Angelo finite type such that the Levi form has corank at most $1$ at $\xi_{0}$ and  if there exists a sequence $\{\eta_j\}\subset\Omega$ such that $\lim\limits_{j\to\infty}\eta_j=\xi_0$ and $\lim\limits_{j\to\infty}s_{\Omega}(\eta_j)=1$ or $\lim\limits_{j\to\infty}h_{\Omega}(\eta_j)=0$, then $\partial \Omega$ is strongly pseudoconvex at $\xi_0$.
\end{theorem}
As a consequence, we obtain the following well-known result (see \cite{JK18, MV19,BK19}). 
\begin{corollary}
Let $\Omega$ be a bounded domain in $\mathbb C^n$ with smooth pseudoconvex boundary. If $\xi_0$ is a boundary point of $\Omega$ of D'Angelo finite type such that the Levi form has corank at most $1$ at $\xi_0$ and if $\lim\limits_{\Omega\ni z\to \xi_0}s_{\Omega}(z)=1$ or $\lim\limits_{\Omega\ni z\to \xi_0} h_{\Omega}(z)=0$, then $\partial \Omega$ is strongly pseudoconvex at $\xi_0$.
\end{corollary} 

\begin{remark}
It is known that the boundary point $\xi_0$ in our situation is $h$-extendible. Therefore, if $\{\eta_j\}$ converges $\left(\frac{1}{m_1},\ldots, \frac{1}{m_{n-1}}\right)$-nontangentially to $\xi_0$, then $\xi_0$ is strongly pseudoconvex as mentioned above. However, we emphasize here that $\{\eta_j\}\subset\Omega$ is an arbitrary sequence converging to $\xi_0$. For the proof of Theorem \ref{main thm}, as in \cite{JK18} we also utilize the scaling method by Pinchuk to show that the complex unit ball $\mathbb B^n$ is biholomorphically equivalent to a model 
\[
M_P=\left\{(z_1,\ldots,z_n)\in \mathbb{C}^{n}\colon \mathrm{Re}(z_n) +P(z_1,\bar{z}_1)+\sum_{\alpha=2}^{n-1}|z_{\alpha}|^2<0\right\},
\]
where $P$ is a non-zero real-valued subharmonic polynomial of degree $2m$, where $2m$ is the D'Angelo type of $\partial\Omega$ at $\xi_0$. Then, this yields $2m=2$ and hence our theorem follows.
\end{remark}

The organization of the paper is described as follows: For the convenience of the reader, we exploit a constructive procedure of the scaling sequence in higher dimension in Section \ref{section 2}, based on the results in \cite{Cho94} and \cite{DN09}. Then we investigate the normality of our scaling sequence which is crucial in determining the fact that $\mathbb B^n$ and $M_P$ are biholomorphically equivalent. We finalize the proof of Theorem \ref{main thm} in Section \ref{proof of the main thm}, after applying a technical lemma \cite[Lemma~$3.2$]{Ber94} related to the biholomorphic equivalence among models.

\section{The scaling sequence in higher dimension}\label{section 2}
This section is devoted to a proof of the normality of our scaling sequence. Then, by using this normality result the biholomorphic equivalence between $M_P$ and the complex unit ball $\mathbb B^n$ will be shown. 

First of all, we recall the following definition which will be used for the proof in this section (see \cite{GK} or \cite{DN09}).
\begin{define} Let $\{\Omega_j\}_{j=1}^\infty$ be a sequence of open sets in $\mathbb C^n$ and $\Omega_0 $ be an open set of $\mathbb C^n$. The sequence $\{\Omega_j\}_{j=1}^\infty$ is said to converge to $\Omega_0 $ (written $\lim\Omega_j=\Omega_0$) if and only if 
\begin{enumerate}
\item[(i)] For any compact set $K\subset \Omega_0,$ there is an $j_0=j_0(K)$ such that $j\geq j_0$ implies that $K\subset \Omega_i$, and 
\item[(ii)] If $K$ is a compact set which is contained in $\Omega_i$ for all sufficiently large $j,$ then  $K\subset \Omega_0$.
\end{enumerate}  
\end{define}

Throughout this section, the domain $\Omega$ and the boundary point $\xi_0\in \partial \Omega $ are assumed to satisfy the hypothesis of Theorem \ref{main thm}. Let $2m$ be the D'Angelo type of $\partial \Omega$ at $\xi_0$. Without loss of generality, we may assume that $\xi_0=0\in\mathbb C^n$ and the rank of Levi form at $\xi_0$ is exactly $n-2$. Let $\rho$ be a smooth defining function for  $\Omega $. After a linear change of coordinates, we can find the coordinate functions $z_1,\ldots, z_n$ defined on a neighborhood $U_0$ of $\xi_0$ such that 
\begin{equation*}
\begin{split}
\rho(z)&=\mathrm{Re}(z_n)+ \sum_{\substack{j+k\leq 2m\\
 j,k>0}} a_{j,k}z_1^j \bar z_1^k\\
&+\sum_{\alpha=2}^{n-1}|z_\alpha|^2+ \sum_{\alpha=2}^{n-1} \sum_{\substack{j+k\leq m\\
 j,k>0}}\mathrm{Re} (( b^\alpha_{j,k}z_1^j \bar z_1^k)z_\alpha)\\
&+O(|z_n| |z|+|z^*|^2|z|+|z^*|^2|z_1|^{m+1}+|z_1|^{2m+1}),
\end{split}
\end{equation*}
where $z=(z_1,\ldots,z_n)$, $z^*=(0,z_2,\ldots,z_{n-1},0)$, and $a_{j,k}, b^\alpha_{j,k}~(2\leq \alpha \leq n-1)$ are $\mathcal{C}^\infty$-smooth functions in a small neighborhood of the origin in $\mathbb C^n$. 

By \cite[Proposition~$2.2$]{Cho94} (see also~\cite[Proposition~$3.1$]{DN09}),  for each point $\eta$ in a small neighborhood of the origin, there exists a unique biholomorphism $\Phi_\eta$ of $\mathbb C^n$, $z=\Phi^{-1}_{\eta}(w)$, such that
\begin{equation}\label{Eq19} 
\begin{split}
\rho(\Phi_{\eta}^{-1}(w))-\rho(\eta)&= \mathrm{Re}(w_n)+ \sum_{\substack{j+k\leq 2m\\
 j,k>0}} a_{j,k}(\eta)w_1^j \bar w_1^k\\
&+\sum_{\alpha=2}^{n-1}|w_\alpha|^2+ \sum_{\alpha=2}^{n-1} \sum_{\substack{j+k\leq m\\
 j,k>0}}\mathrm{Re} [(b^\alpha_{j,k}(\eta)w_1^j \bar w_1^k)w_\alpha]\\
&+O(|w_n| |w|+|w^*|^2|w|+|w^*|^2|w_1|^{m+1}+|w_1|^{2m+1}),
\end{split}
\end{equation}
where $w^*=(0,w_2,\ldots,w_{n-1},0)$.

Now let us denote by 
\begin{equation}\label{Eq5}
\begin{split} 
A_l(\eta)&=\max \{|a_{j,k}(\eta)|: \ j+k=l\}\ (  2\leq l \leq 2m),\\
B_{l'}(\eta)&=\max \{|b^\alpha_{j,k}(\eta)|: \ j+k=l',\ 2\leq \alpha\leq n-1\}\ (  2\leq l' \leq m).
\end{split}
\end{equation}
For each $\delta>0$, we define $\tau(\eta,\delta)$ as follows. 
\[
\tau(\eta,\delta)=\min \left\{\big( \delta/A_l(\eta) \big)^{1/l},\ \big( \delta^{\frac{1}{2}}/B_{l'}(\eta) \big)^{1/{l'}}: \ 2\leq l \leq 2m,\ 2\leq l' \leq m \right\}.
\]
We note that the D'Angelo type of $\partial \Omega$ at $\xi_0$ equals $2m$ and the Levi form has rank at least $n-2$ at $\xi_0$. Therefore, $A_{2m}(\xi_0)\ne 0$ and  hence there exists a sufficiently small neighborhood $U$ of  $\xi_0$ such that $|A_{2m}(\eta)|\geq c>0$ for all $\eta\in U$. This yields the relation
\begin{equation}\label{Eq7} 
\delta^{1/2}\lesssim \tau(\eta,\delta)  \lesssim \delta^{1/(2m)}\  (\eta\in U).
\end{equation}Let us define an anisotropic dilation $\Delta_\eta^\epsilon$ by 
\[
\Delta_\eta^\epsilon (w_1,\ldots,w_n)=\left(\frac{w_1}{\tau_1(\eta,\epsilon)},\ldots,\frac{w_n}{\tau_n(\eta,\epsilon)}\right),
\]
where $\tau_1(\eta,\epsilon)=\tau(\eta,\epsilon),\  \tau_k(\eta,\epsilon)=\sqrt{\epsilon}\ (2\leq k\leq n-1), \ \tau_n(\eta,\epsilon)=\epsilon$. For each $\eta\in \partial \Omega$, if we set $\rho_\eta^\epsilon(w)=\epsilon^{-1}\rho\circ \Phi_\eta^{-1}\circ(\Delta_\eta^\epsilon)^{-1}(w)$, then \eqref{Eq19} and \eqref{Eq7} imply that
\begin{equation}\label{Eq20} 
\begin{split}
\rho_\eta^\epsilon(w)&= \mathrm{Re}(w_n)+ \sum_{\substack{j+k\leq 2m\\
 j,k>0}} a_{j,k}(\eta) \epsilon^{-1} \tau(\eta,\epsilon)^{j+k}w_1^j \bar w_1^k+\sum_{\alpha=2}^{n-1}|w_\alpha|^2\\
&+ \sum_{\alpha=2}^{n-1} \sum_{\substack{j+k\leq m\\
 j,k>0}}\mathrm{Re} ( b^\alpha_{j,k}(\eta)\epsilon^{-1/2} \tau(\eta,\epsilon)^{j+k}w_1^j \bar w_1^kw_\alpha)+O(\tau(\eta,\epsilon)).
\end{split}
\end{equation}

In what follows, let us fix a sufficiently small neighborhood $U_0$ of $\xi_0$ and let $\{\eta_j\}\subset \Omega$ be a sequence converging to $\xi_0$. Further, we may also assume that $\eta_j\in U_0^-:=U_0\cap\{\rho<0\}$ for all $j$. For this sequence $\{\eta_j\}$, one associates with a sequence of points $\eta_j'=(\eta_{1j}, \ldots,\eta_{(n-1)j}, \eta_{nj}+\epsilon_j)$, $ \epsilon_j>0$, $\eta_j'$ in the hypersurface $\{\rho=0\}$. Let us consider the sequence of dilations $\Delta_{\eta_j'}^{\epsilon_j}$. Then $\Delta_{\eta_j'}^{\epsilon_j}\circ \Phi_{\eta_j'}({\eta}_j)=(0,\ldots,0,-1)$ and moreover it follows from \eqref{Eq20} that $\Delta_{\eta_j'}^{\epsilon_j}\circ \Phi_{\eta_j'}(\{\rho=0\}) $ is defined by
\[
\mathrm{Re}(w_n)+ P_{\eta_j'}(w_1,\bar w_1)+\sum_{\alpha=2}^{n-1}|w_\alpha|^2+ \sum_{\alpha=2}^{n-1}\mathrm{Re}(Q^\alpha_{\eta_j'}(w_1,\bar w_1)w_\alpha)+O(\tau(\eta_j',\epsilon_j))=0,
\]
where
\begin{equation*}
\begin{split}
&P_{\eta_j'}(w_1,\bar w_1):=\sum_{\substack{j+k\leq 2m\\
 j,k>0}} a_{j,k}(\eta_j') \epsilon_j^{-1} \tau(\eta_j',\epsilon_j)^{j+k}w_1^j \bar w_1^k,\\
&Q^\alpha_{\eta_j'}(w_1,\bar w_1):= \sum_{\substack{j+k\leq m\\
 j,k>0}} b^\alpha_{j,k}(\eta_j')\epsilon_j^{-1/2} \tau(\eta_j',\epsilon_j)^{j+k}w_1^j \bar w_1^k.
\end{split}
\end{equation*}

Then one can deduce from \eqref{Eq5} that the coefficients of $P_{\eta_j'}$ and $Q^\alpha_{\eta_j'}$ are bounded by one. Therefore, after taking a subsequence, we may assume  that $\{P_{\eta_j'}\}$ converges uniformly on every compact subset of $\mathbb C$ to a polynomial $P(z_1,\bar z_1)$.  Moreover, $\{Q^\alpha_{\eta_j'}\}~(2\leq \alpha\leq  n-1)$ converge uniformly on every compact subset of $\mathbb C$ to $0$ by the following lemma.
\begin{lemma}[see Lemma~$2.4$ in \cite{Cho94}]\label{estimate of Q} $|Q^\alpha_{\eta_j'}(w_1,\bar w_1)|\leq \tau(\eta_j',\epsilon_j)^{\frac{1}{10}} $ for all $\alpha=2,\ldots, n-1$ and $|w_1|\leq 1$, provided that $\tau$ is sufficiently small.
\end{lemma}
Then, by Lemma~\ref{estimate of Q}, after taking a subsequence, one can deduce that $\Delta_{\eta_j'}^{\epsilon_j}\circ \Phi_{\eta_j'}(U_0^-)$ converges to the following model
\begin{equation}\label{Eq29} 
M_P:=\left\{\hat\rho:=\mathrm{Re}(w_n)+ P(w_1,\bar w_1)+\sum_{\alpha=2}^{n-1}|w_\alpha|^2<0\right\},
\end{equation}
where $P(w_1,\bar w_1)$ is a polynomial of degree $\leq 2m$ without harmonic terms (cf.~\cite[p.~$153$]{DN09}).
\begin{remark}\label{remark1}
It is well-known that $M_P$ is a smooth limit of the pseudoconvex domains $ \Delta_{\eta_j'}^{\epsilon_j}\circ \Phi_{\eta_j'}(U_0^-) $. Then, $M_P$ becomes to be a pseudoconvex domain. Therefore, the function $\hat\rho$ in (\ref{Eq29}) is plurisubharmonic, and thus $P$ is a subharmonic polynomial whose Laplacian does not vanish identically. 
\end{remark}

Now let us recall the following theorem, which ensures the normality of the scaling sequence that will be given in the proof of Proposition \ref{biholomorphic equivalence to the unit ball}.

\begin{proposition}[see Theorem~$3.11$ in \cite{DN09}]\label{Normality of the scaling seq.}
Let $\Omega$ be a domain in $\mathbb C^n$. Suppose that $\partial \Omega$ is pseudoconvex, of D'Angelo finite type and is $\mathcal{C}^\infty$-smooth near a boundary point $(0,\ldots,0)\in\partial \Omega$. Suppose that the Levi form has corank at most $1$ at $(0,\ldots,0)$. Let $D$ be a domain in $\mathbb C^k$ and $\varphi_j:D\to \Omega$ be a sequence of holomorphic mappings such that $\eta_j:=\varphi_j(a)$ converges to $(0,\ldots,0)$ for some point $a\in D$. Let $\{T_j\}$ be a sequence of automorphisms of $\mathbb  C^n$ which associates with the sequence $\{\eta_j\}$ by the  method of the dilation of coordinates (i.e.,\ $T_j=\Delta_{\eta_j'}^{\epsilon_j}\circ \Phi_{\eta_j'}$). Then $\{T_j\circ \varphi_j\}$ is normal and its limits are holomorphic mappings from $D$ to the domain of the form 
\[
M_P=\left\{  (w_1,\ldots,w_n)\in \mathbb C^n:\mathrm{Re}(w_n)+P(w_1,\bar w_1)+\sum_{\alpha=2}^{n-1}|w_\alpha|^2<0\right \},
\]
where $P\in\mathcal{P}_{2m}$. Here $\mathcal{P}_{2m}$ denotes the space of real-valued polynomials on $\mathbb C$ of degree $\leq 2m$ without harmonic terms.
\end{proposition}

\begin{proposition}\label{biholomorphic equivalence to the unit ball}  
$M_P$ is biholomorphically equivalent to the complex unit ball $\mathbb B^n$.
\end{proposition}
\begin{proof} Let $\{\eta_j\}\subset \Omega$ be a sequence as in Theorem \ref{main thm}, that is, $\eta_j \to\xi_{0}=0$ as $j\to \infty$. We now split the proof into two following cases:

\noindent
{\bf Case 1:} {\it $\lim\limits_{j\to \infty} s_{\Omega}(\eta_j)=1$.} Let us set $\delta _j=2(1-s_{\Omega}(\eta_j))$ for all $j$. Then by our assumption, for each $j$, there exists an injective holomorphic map $f_j:\Omega\to \mathbb{B}^n $ such that $f_j(\eta_j)=(0',0)$ and $\mathbb{B}(0;1-\delta _j)\subset f_j(\Omega)$. Then by \cite[Proposition $2.2$]{DN09} and the hypothesis of Theorem \ref{main thm}, after choosing a suitable sequence of injective holomorphic mappings $f_j: \Omega\to \mathbb B^n$ whose existence is assured by the assumption on the squeezing function $s_{\Omega}$, for each compact subset $K\Subset \mathbb B^n$ and each neighborhood $U_0$ of $\xi_0$, there exists an integer $j_0$ such that $f_j^{-1}(K)\subset \Omega\cap U_0$ for all $j\geq j_0$, i.e., $f_j(\Omega \cap U_0)$ converges to $\mathbb B^n$. Then it follows from Proposition~\ref{Normality of the scaling seq.} that the sequence $T_j\circ f_j^{-1} \colon f_j(\Omega \cap U_0) \to  T_j(\Omega \cap U_0) $ is normal and its limits are holomorphic mappings from $\mathbb B^n$ to $M_P$. Moreover, by Montel's theorem the sequence $ f_j\circ T_j^{-1} \colon T_j(\Omega \cap U_0)\to   f_j(\Omega \cap U_0) \subset \mathbb B^n$ is also normal. We further note that the sequence $T_j\circ f_j^{-1}$ is not compactly divergent since $T_j\circ f_j^{-1}(0',0)=(0',-1)$. Then by \cite[Proposition $2.1$]{DN09}, after taking some subsequence of $\{T_j\circ f^{-1}_{j}\}$, we may assume that such a subsequence converges uniformly on every compact subset of $\mathbb B^n$ to a biholomorphism $F$ from $\mathbb B^n$ onto $M_P$, as desired.

\noindent
{\bf Case 2:} {\it $\lim\limits_{j\to \infty} h_{\Omega}(\eta_j)=0$.}  

Since the point $\xi_0$ is a local peak point (cf. \cite{Yu94}), by \cite[Proposition $3.4$]{MV12}, one has $\lim\limits_{j\to \infty} h_{U_0\cap \Omega}(\eta_j)=0$. Moreover, by our assumption, there exist a sequence of positive real numbers $R_j\to +\infty$ and a sequence of biholomorphic embeddings $g_j:\mathbb B^n\to U_0\cap \Omega$ such that $g_j(0)=\eta_j$ and $B_{U_0\cap \Omega}(\eta_j,R_j)\subset g_j(\mathbb B^n)$. Then it follows from Proposition~\ref{Normality of the scaling seq.} that the sequence $ T_j\circ g_j\colon \mathbb B^n \to  T_j(\Omega \cap U_0) $ is normal and its limits are holomorphic mappings from $\mathbb B^n$ to $M_P$. Moreover, by Montel's theorem the sequence $ g_j^{-1}\circ T_j^{-1} \colon T_j(\Omega \cap U_0)\to   g_j^{-1}(\Omega \cap U_0) \subset \mathbb B^n$ is also normal. We also note that the sequence $T_j\circ g_j$ is not compactly divergent since $T_j\circ g_j(0',0)=(0',-1)$. Then by \cite[Proposition $2.1$]{DN09}, after taking some subsequence of $\{T_j\circ g_j\}$, we may assume that such a subsequence converges uniformly on every compact subset of $\mathbb B^n$ to a biholomorphism $G$ from $\mathbb B^n$ onto $M_P$, as desired.

Altogether, the proof is now complete.
\end{proof}
\begin{remark}\label{final remark in section 3}As in \cite{JK18}, the sequence $\{\eta_j\}$ can be chosen so that $\eta_j$ converges to $\xi_0$ along the direction normal to the boundary. Therefore, $P(z_1,\bar{z}_1)$ must be homogeneous subharmonic polynomial of degree $2m$. However, by using the argument as in \cite[Sections $3$ and $4$ ]{Ber94} (see also \cite[Section $4$]{DN09}), in our situation, $P$ is also a  homogeneous subharmonic polynomial of degree $2m$ without harmonic terms. Moreover, one sees from Remark \ref{remark1} in particular that $\Delta P\not \equiv 0$.
\end{remark}

\section{Proof of the main theorem}\label{proof of the main thm}  
We shall complete the proof of Theorem ~\ref{main thm} as our main result in this section. Recall from Remark~\ref{final remark in section 3} that 
\[
M_P=\left\{(z_1,\ldots,z_n)\in \mathbb{C}^{n}\colon \mathrm{Re}(z_n) +P(z_1,\bar{z}_1)+|z_2|^2+\cdots+|z_{n-1}|^2<0\right\},
\]
where $P$ is a non-zero real-valued subharmonic polynomial of degree $2m$. We define a space $\mathcal{H}_{2m}$ by setting
\[
\mathcal{H}_{2m}:=\{H\in\mathcal{P}_{2m}: \mathrm{deg}H=2m,\; H\;\text{is homogeneous and subharmonic}\}, 
\]where the space $\mathcal{P}_{2m}$ is given as in Proposition~\ref{Normality of the scaling seq.}.

With these notations, we prepare one more lemma in order to prove Theorem~\ref{main thm}. 
\begin{lemma}[see Lemma $3.2$ in \cite{Ber94}]\label{equivalence of models} Let $Q\in \mathcal{P}_{2m}$ and $H\in\mathcal{H}_{2m}$. If $M_Q$ and $M_H$ are biholomorphically equivalent, then the homogeneous part of higher degree in $Q$ is equal to $\lambda H(e^{i\nu}z)$ for some $\lambda >0$ and $\nu\in [0,2\pi]$.
\end{lemma}
We note first that the complex unit ball $\mathbb{B}^n$ is biholomorphic to the Siegel half-space $\{(z_1,\ldots,z_n)\in \mathbb{C}^n \colon\mathrm{Re}(z_n) +|z_1|^2+|z_2|^2+\cdots+|z_{n-1}|^2<0\}$. In addition, Proposition \ref{biholomorphic equivalence to the unit ball} and Lemma \ref{equivalence of models} imply that $P(z_1,\bar {z}_1)=c|z_1|^{2}$ for some $c>0$, that is, $m=1$. Combining these two facts, we conclude that $\Omega$ is strongly pseudoconvex at $\xi_0$ ($\xi_0$ is of the D'Angelo type $2$), which ends the proof of Theorem \ref{main thm}.
 
\begin{Acknowledgement}Part of this work was done while the authors were visiting the Vietnam Institute for Advanced Study in Mathematics (VIASM). They would like to thank the VIASM for financial support and hospitality. This research was supported by the Vietnam National Foundation for Science and Technology Development (NAFOSTED) under grant number 101.02-2017.311 and the National Research Foundation of the Republic of Korea under grant number NRF-2018R1D1A1B07044363.
\end{Acknowledgement}


\begin{thebibliography}{99}

\bibitem[Ber94]{Ber94} F. Berteloot, \textit{Characterization of models in $\CC^2$ by their automorphism groups}, Internat. J. Math. 5 (1994), no. 5, 619--634.
\bibitem[BK19]{BK19} D. Borah and D. Kar, \textit{Boundary behaviour of the Carath\'{e}odory and Kobayashi-Eisenman volume elements}, arXiv:1902.10022v2.
\bibitem[Cho94]{Cho94} S. Cho, \textit{Boundary behavior of the Bergman kernal function on some pseudoconvex domains in $\mathbb C^n$}, Trans. Amer. Math. Soc. 345 (1994), no. 2, 803--817.
\bibitem[DGZ12]{DGZ12} F. Deng, Q. Guan and L. Zhang, \textit{Some properties of squeezing functions on bounded domains}, Pacific J. Math. 257 (2012), no. 2, 319--341. 
\bibitem[DGZ16]{DGZ16} F. Deng, Q. Guan and L. Zhang, \textit{Properties of squeezing functions and global transformations of bounded domains}, Trans. Amer. Math. Soc. 368 (2016), no. 4, 2679--2696. 
\bibitem[DFW14]{DFW14} K. Diederich, J. E. Forn{\ae}ss and E. F. Wold, \textit{Exposing points on the boundary of a strictly pseudoconvex or a locally convexifiable domain of finite $1$-type}, J. Geom. Anal. 24 (2014), no. 4, 2124--2134.
\bibitem[DN09]{DN09} D. T. Do and V. T. Ninh, \textit{Characterization of domains in $\mathbb  C^n$ by their noncompact automorphism groups}, Nagoya Math. J. 196 (2009), 135--160.
\bibitem[GK87]{GK}  R. E. Greene and S. G. Krantz, \textit{Biholomorphic self-maps of domains,} Lecture Notes in Math., 1276 (1987), 136--207.
\bibitem[JK18]{JK18} S. Joo and K.-T. Kim, \textit{On boundary points at which the squeezing function tends to one}, J. Geom. Anal. 28 (2018), no. 3, 2456--2465.
\bibitem[KZ16]{KZ16} K.-T. Kim and L. Zhang, \textit{On the uniform squeezing property and the squeezing function}, Pacific. J. Math. 282 (2016), no. 2,  341--358.
\bibitem[MV12]{MV12}  P. Mahajan and K. Verma, \textit{Some aspects of the Kobayashi and Carathéodory metrics on pseudoconvex domains}, J. Geom. Anal. 22 (2012), no. 2, 491--560. 
\bibitem[MV19]{MV19} P. Mahajan and K. Verma, \textit{A comparison of two biholomorphic invariants}, Internat. J. Math. 30 (2019), no. 1, 1950012, 16pp.
\bibitem[Ni18]{Ni18} N. Nikolov, \textit{Behavior of the squeezing function near h-extendible boundary points}, Proc. Amer. Math. Soc. 146 (2018), no. 8, 3455--3457.
\bibitem[NV18]{NV18} N. Nikolov and K. Verma, \textit{On the squeezing function and Fridman invariants}, arXiv:1810.10739v1.
\bibitem[NN19]{NN19} V. T. Ninh and Q. D. Nguyen, \textit{Some properties of $h$-extendible domains in $\mathbb C^{n+1}$}, arXiv:1907.00152.
\bibitem[Yu94]{Yu94} J. Yu, \textit{Peak functions on weakly pseudoconvex domains}, Indiana Univ. Math. J. 43 (1994), no. 4, 1271--1295. 

\end{thebibliography}
\end{document}